\documentclass[12pt,a4paper]{article}
\usepackage{latexsym,amssymb,amsfonts,amsmath,amsthm,nccmath,enumitem,mathtools,setspace,graphicx,bm}
\usepackage[dvipsnames]{xcolor}
\usepackage[hidelinks]{hyperref}
\usepackage[margin=2cm]{geometry}
\usepackage[noblocks]{authblk}
\newtheorem{theorem}{Theorem}[section]
\newtheorem{corollary}[theorem]{Corollary}
\newtheorem{lemma}[theorem]{Lemma}

\theoremstyle{definition}

\newtheorem{remark}[theorem]{Remark}

\def \leq {\leqslant}
\def \geq {\geqslant}

\def \leq {\leqslant}
\def \geq {\geqslant}
\def \le {\leqslant}
\def \ge {\geqslant}

\def \mod{\pmod}

\let\oldproofname=\proofname
\renewcommand{\proofname}{\rm\bf{\oldproofname}}

\title{On decomposition thresholds for odd-length cycles and other tripartite graphs}
\date{}
\author{Darryn Bryant\thanks{School of Mathematics and Physics, The University of Queensland, St Lucia QLD 4067, Australia.}\qquad Peter Dukes\thanks{Department of Mathematics and Statistics, University of Victoria, Canada}\qquad Daniel Horsley\thanks{School of Mathematics, Monash University, Melbourne VIC 3800, Australia.} \\ \vspace*{-2mm} Barbara Maenhaut$^*$\qquad Richard Montgomery\thanks{Mathematics Institute, University of Warwick, Coventry, CV4 7AL, UK}}

\begin{document}
\maketitle
\setstretch{1.2}

\begin{abstract}
An (edge) \emph{decomposition} of a graph $G$ is a set of subgraphs of $G$ whose edge sets partition the edge set of $G$. Here we show, for each odd $\ell \geq 5$, that any graph $G$ of sufficiently large order $n$
with minimum degree at least $\bigl(\frac{1}{2}+\frac{1}{2\ell-4}+o(1)\bigr)n$ has a decomposition into $\ell$-cycles if and only if $\ell$ divides $|E(G)|$ and each vertex of $G$ has even degree. This threshold cannot be improved beyond $\frac{1}{2}+\frac{1}{2\ell-2}$. It was previously shown that the thresholds approach $\frac{1}{2}$ as $\ell$ becomes large, but our thresholds do so significantly  more rapidly. Our methods can be applied to tripartite graphs more generally and we also obtain some bounds for decomposition thresholds of other tripartite graphs.
\end{abstract}

\section{Introduction}\label{S:intro}

When we refer to a graph in this paper we always mean a simple undirected graph. Let $F$ and $G$ be graphs. A \emph{decomposition} of $G$ is a set of subgraphs of $G$ whose edge sets partition the edge set of $G$. If each graph in such a decomposition is isomorphic to a graph $F$, we say it is an \emph{$F$-decomposition.} Let $\gcd(G)$ denote the largest integer that divides the degree of each vertex of $G$. We say that $G$ is \emph{$F$-divisible} if $|E(F)|$ divides $|E(G)|$ and $\gcd(F)$ divides $\gcd(G)$. The \emph{$F$-decomposition threshold} $\delta_F$ is the infimum of all nonnegative real numbers $\delta$ with the property that, for each sufficiently large integer $n$, every $F$-divisible $n$-vertex graph with minimum degree at least $\delta n$ has an $F$-decomposition. For each integer $\ell \geq 3$, let $C_\ell$ represent the cycle of length $\ell$. Barber, K\"{u}hn, Lo and Osthus \cite{BarKuhLoOst} showed that $\delta_{C_4} = \frac{2}{3}$ and $\delta_{C_\ell} = \frac{1}{2}$ for all even $\ell \geq 6$. Taylor \cite{Tay} obtained exact, rather than asymptotic, versions of these results for all even $\ell \neq 6$. For odd-length cycles only bounds on the decomposition threshold are known. It is not difficult to show that $\delta_{C_\ell} \geq \frac{1}{2}+\frac{1}{2\ell-2}$ for each odd $\ell \geq 3$. This can be seen by considering an $n$-vertex regular graph that contains a balanced complete bipartite subgraph and has valency just under $(\frac{1}{2}+\frac{1}{2\ell-2})n$ (see Lemma~\ref{L:lowerBound} for details). The best known upper bound for $\ell=3$ is due to Delcourt and Postle \cite{DelPos} who showed that $\delta_{C_3} \leq (7+\sqrt{21})/14 \lessapprox 0.82733$. By a result of \cite{GloKuhLoMonOst}, this upper bound applies for all odd-length cycles. A special case of a result of Joos and K\"{u}hn \cite{JooKuh} on decompositions of hypergraphs into tight cycles shows that $\delta_{C_\ell} \leq \frac{1}{2}+O(\ell^{-1/8!})$ for each odd $\ell \geq 3$ and hence that $\delta_{C_\ell}$ approaches $\frac{1}{2}$ as $\ell$ becomes large. Here we show that $\delta_{C_\ell} \leq \frac{1}{2}+\frac{1}{2\ell-4}$ for each odd $\ell \geq 5$.

\begin{theorem}\label{T:main}
Let $\ell \geq 5$ be an odd integer and let $\epsilon>0$ be a real number. There is an integer $n_0 \coloneqq n_0(\ell,\epsilon)$ such that every $C_\ell$-divisible graph of order $n \geq n_0$ with minimum degree at least $(\frac{1}{2}+\frac{1}{2\ell-4}+\epsilon)n$ has a $C_\ell$-decomposition.
\end{theorem}

While $\frac{1}{2}+\frac{1}{2\ell-4}$ is very unlikely to be best possible, as we noted above, it cannot be improved beyond $\frac{1}{2}+\frac{1}{2\ell-2}$. We prove Theorem~\ref{T:main} through combining the seminal results in \cite{BarKuhLoOst} with two subsidiary results of our own.

We first require some definitions and notation. A \emph{weighted graph} $W$ consists of an \emph{underlying graph} $U$, together with an assignment $w_W$ of positive weights to the edges of $U$. We say a weighted graph $W'$ is a \emph{scaled copy} of $W$ if there is a real number $\alpha$ and an isomorphism $f$ from the underlying graph of $W$ to the underlying graph of $W'$ such that $w_{W'}(e)=\alpha\, w_W(f(e))$ for all $e \in E(W)$. For a set $\mathcal{W}$ of weighted graphs, a \emph{fractional $\mathcal{W}$-decomposition} of a graph $G$ is a finite collection $\mathcal{G^*}$ of scaled copies of weighted graphs in $\mathcal{W}$ such that the underlying graph of each weighted graph in $\mathcal{G}^*$ is a subgraph of $G$ and, for each $e \in E(G)$, the sum of the weights assigned to $e$ by weighted graphs in $\mathcal{G^*}$ is exactly 1. When $\mathcal{W}=\{W\}$ we will denote this as a \emph{fractional $W$-decomposition.}

For a given odd $\ell \geq 5$, our proof of Theorem~\ref{T:main} proceeds by first establishing that $n$-vertex graphs with minimum degree at least $(\frac{1}{2}+\frac{1}{2\ell-4})n$ have fractional decompositions into scaled copies of a certain weighted $K_3$, and then showing that this implies that $\delta_{C_\ell} \leq \frac{1}{2}+\frac{1}{2\ell-4}$. The following theorem supplies the first piece of this argument.

\begin{theorem}\label{T:weightedTriDecomp}
Let $T_{e_1,e_2,e_3}$ be the weighted $K_3$ in which the edges have weights $e_1 \ge e_2 \ge e_3$. Every graph of order $n$ with minimum degree at least $\delta(e_1,e_2,e_3)\,n$ has a fractional $(T_{e_1,e_2,e_3})$-decomposition, where
\begin{equation*}\label{E:deltaDef}
    \delta(e_1,e_2,e_3) \coloneqq \mfrac{1}{2}+\max\left\{\mfrac{e_3}{2e_1+2e_2-2e_3},\mfrac{e_2+e_3}{8e_1-2e_2-2e_3}\right\}.
\end{equation*}
\end{theorem}

The weighted $K_3$ that will be of interest for bounding $\delta_{C_\ell}$ is $T_{\ell-2,1,1}$. Theorem~\ref{T:weightedTriDecomp} implies that, for each integer $\ell \geq 3$, every graph of order $n$ with minimum degree at least $(\frac{1}{2}+\frac{1}{2\ell-4})n$ has a fractional $(T_{\ell-2,1,1})$-decomposition. There is no reason to think that the bound of Theorem~\ref{T:weightedTriDecomp} is best possible. However, we show in Lemmas~\ref{L:lowerBound} and \ref{L:lowerBound2} that it cannot be reduced below $\frac{1}{2}+\max\{\frac{e_3}{2e_1+2e_2},\frac{e_2+e_3}{8e_1+2e_2+2e_3}\}$.

For a positive real number $\eta$ and a graph $G$ with $n$ vertices, an \emph{$\eta$-approximate decomposition} of $G$ is a decomposition of a subgraph $G'$ of $G$ such that $|E(G')| \geq |E(G)| - \eta n^2$. We are able to use Theorem~\ref{T:weightedTriDecomp} to establish results concerning (integral) decompositions into odd-length cycles because a fractional $(T_{\ell-2,1,1})$-decomposition of a sufficiently large graph can be transformed into an approximate $C_{\ell}$-decomposition of the same host graph. This is a special case of a more general result that we state below as Theorem~\ref{T:approximateCycleDecomp}. For a positive integer $k$, a \emph{proper $k$-colouring} of a graph $F$ is a function $c:V(F) \rightarrow \{1,\ldots,k\}$ such that $c(x) \neq c(y)$ for each edge $xy$ of $F$. The sets $c^{-1}(i)$ for $i \in \{1,\ldots,k\}$ (some possibly empty) are the \emph{colour classes} of $c$. Let $F$ be a graph, let $c$ be a proper $k$-colouring of $F$ and, for all $2$-subsets $\{i,j\}$ of $\{1,\ldots,k\}$, let $e_{ij}$ be the number of edges of $F$ that are incident with both a vertex coloured $i$ and a vertex coloured $j$. The \emph{condensation} of $F$ with respect to $c$ is the weighted graph on vertex set $\{1,\ldots,k\}$ in which the edge $ij$ is absent if $e_{ij}=0$ and otherwise is present with weight $e_{ij}$.

\begin{theorem}\label{T:approximateCycleDecomp}
Let $F$ be a graph, let $\mathcal{W}$ be a set of condensations of $F$, and let $\eta>0$ be a real number. There is an integer $n_0 \coloneqq n_0(F,\eta)$ such that each graph $G$ of order at least $n_0$ that has a fractional $\mathcal{W}$-decomposition also has an $\eta$-approximate $F$-decomposition.
\end{theorem}

In this paper we will only apply Theorem~\ref{T:approximateCycleDecomp} in the case where $\mathcal{W}$ contains only one weighted graph, but we believe the more general formulation we prove might be useful in future work. For odd $\ell \geq 3$, we can consider a proper 3-colouring of $C_\ell$ in which two colour classes contain $\frac{\ell-1}{2}$ vertices and the third part contains only one vertex. So $T_{\ell-2,1,1}$ is indeed a condensation of $C_\ell$.

We are able to obtain results on decompositions, rather than simply on approximate decompositions, thanks to results in \cite{BarKuhLoOst}. For any $\eta>0$ we define $\delta^{\eta}_F$ to be the infimum of all nonnegative real numbers $\delta$ with the property that, for each sufficiently large integer $n$, every $n$-vertex graph with minimum degree at least $\delta n$ has an $\eta$-approximate $F$-decomposition. The \emph{approximate $F$-decomposition threshold} $\delta^{0+}_F$ is the supremum of $\delta^{\eta}_F$ over all positive real numbers $\eta$. Using `iterative absorption' techniques, the following is proved in \cite{BarKuhLoOst}.

\begin{theorem}[\cite{BarKuhLoOst}]\label{T:approxToIntegral}
For each odd integer $\ell \geq 3$, we have $\delta_{C_\ell}=\delta^{0+}_{C_\ell}$.
\end{theorem}

Recall that $T_{\ell-2,1,1}$ is a condensation of $C_\ell$ for each odd $\ell \geq 3$. This means that, between them, Theorems~\ref{T:weightedTriDecomp} and \ref{T:approximateCycleDecomp} imply that $\delta^{0+}_{C_\ell} \leq \frac{1}{2}+\frac{1}{2\ell-4}$. Thus $\delta_{C_\ell} \leq \frac{1}{2}+\frac{1}{2\ell-4}$ by Theorem~\ref{T:approxToIntegral} and Theorem~\ref{T:main} immediately follows.

In Section~\ref{S:defs} we introduce some more definitions and notation and in Section~\ref{S:LBs} we prove some lower bounds on decomposition thresholds that imply the lower bounds we have already mentioned. Next, in Section~\ref{S:weightedTriDecomps}, we consider fractional decompositions of graphs into weighted triangles and prove Theorem~\ref{T:weightedTriDecomp}. Section~\ref{S:condensations} is devoted to proving Theorem~\ref{T:approximateCycleDecomp} using regularity lemma based arguments similar to those in \cite{HaxRod} and \cite{Yus2005}. We examine the implications of our results for decompositions into tripartite graphs more generally in Section~\ref{S:tripartite} before giving some concluding thoughts in Section~\ref{S:conclusion}.

\section{Definitions and notation}\label{S:defs}

Let $G$ be a graph. For a subset $A$ of $V(G)$ we denote by $G[A]$ the subgraph of $G$ induced by the set $A$. For disjoint subsets $A$ and $B$ of $V(G)$ we denote by $G[A,B]$ the bipartite subgraph of $G$ induced by the  bipartition $\{A,B\}$.

Recall that a \emph{weighted graph} $W$ consists of an \emph{underlying graph} $U$, together with an assignment $w_W$ of positive weights to the edges of $U$. Note that $U$ need not be complete and we do not allow edges of weight 0. We write $V(W)$ and $E(W)$ for $V(U)$ and $E(U)$. Unless otherwise stated we will assume the assignment of weights to the edges of a weighted graph $W$ is called $w_W$. We write $\Vert W \Vert$ for the total weight $\sum_{e \in E(W)}w_W(e)$ of $W$.

Let $W$ and $W'$ be weighted graphs. We say $W'$ is a \emph{weighted subgraph} of $W$ if $V(W') \subseteq V(W)$, $E(W') \subseteq E(W)$ and $w_{W'}(e) \leq w_W(e)$ for each $e \in E(W')$. We say $W$ and $W'$ are \emph{similar with scale factor $\alpha$} if there is an isomorphism $f$ from the underlying graph of $W$ to the underlying graph of $W'$ such that $w_{W'}(e)=\alpha\, w_W(f(e))$ for all $e \in E(W)$. We refer to $W'$ as a \emph{scaled copy} of $W$. If $\alpha=1$, the weighted graphs are \emph{congruent} and we refer to $W'$ as a \emph{copy} of $W$. If $f$ is the identity function, we say $W$ and $W'$ are \emph{superimposed}. A graph can, of course, be thought of as a weighted graph all of whose edges have weight 1 and we make no distinction between these objects.

A \emph{fractional packing} of a weighted graph $W$ is a finite collection $\mathcal{P}^*$ of weighted subgraphs of $W$ such that, for each $e \in E(W)$, we have $\sum_{W' \in \mathcal{P}^*(e)}w_{W'}(e) \leq w_W(e)$ where $\mathcal{P}^*(e)=\{W' \in \mathcal{P}^*: e \in E(W')\}$. If each weighted graph in $\mathcal{P}^*$ is a scaled copy of some fixed weighted graph $F$, then $\mathcal{P}^*$ is a  \emph{fractional $F$-packing}. The \emph{leftover} of $\mathcal{P}^*$ is the spanning weighted subgraph $L$ of $W$ such that, for each $e \in E(W)$, we have that $w_L(e)$ is $w_W(e)-\sum_{W' \in \mathcal{P}^*(e)}w_{W'}(e)$ if the latter is positive and $e \notin E(L)$ if it is 0. A fractional packing whose leftover contains no edges is a \emph{fractional decomposition}. Note that we allow fractional packings to contain superimposed scaled copies of a graph, but these can always be removed if desired by amalgamating the weights. We write $\Vert \mathcal{P}^* \Vert$ for $\sum_{W' \in \mathcal{P}^*}\Vert W' \Vert$. For a positive real number $\eta$ and weighted graph $W$ with $n$ vertices, an \emph{$\eta$-approximate fractional decomposition} of $W$ is a fractional packing $\mathcal{P}^*$ of $W$ such that $\Vert \mathcal{P}^* \Vert \geq \Vert W \Vert - \eta n^2$.

For a graph $F$, the \emph{fractional $F$-decomposition threshold} $\delta^*_F$ is defined as the infimum of all nonnegative real numbers $\delta$ with the property that, for each sufficiently large integer $n$, every $n$-vertex $F$-divisible graph with minimum degree at least $\delta n$ has a fractional $F$-decomposition. Note that requiring that the host graphs be $F$-divisible in this definition ensures that $\delta^*_F \leq \delta_F$ because every $F$-decomposition is a fractional $F$-decomposition. For a weighted graph $W$ that is not a graph, we do not define a notion of $W$-divisibility and nor do we define $\delta_W$. We define the \emph{fractional $W$-decomposition threshold} $\delta^*_W$ as the infimum of all nonnegative real numbers $\delta$ with the property that, for each sufficiently large integer $n$, every $n$-vertex graph with minimum degree at least $\delta n$ has a fractional $W$-decomposition.

\section{Lower bounds}\label{S:LBs}

We now present two lemmas which give lower bounds on fractional decomposition thresholds of weighted graphs. These provide a limit on how much the upper bound of Theorem~\ref{T:weightedTriDecomp} can be improved. The first of them also specialises to give the well-known lower bound on the decomposition thresholds of odd-length cycles that we mentioned previously. These lemmas will also be of use in Section~\ref{S:tripartite} where we consider decomposition thresholds of various tripartite graphs. A weighted graph is \emph{empty} if it has no edges.

\begin{lemma}\label{L:lowerBound}
Let $\rho$ be a positive real number and let $W$ be a nonempty weighted graph such that, for every partition $\{U_1,U_2\}$ of $V(W)$, we have $\Vert W[U_1] \cup W[U_2]\Vert \geq \rho \Vert W \Vert$. Then
\[\delta^*_{W} \geq \mfrac{1}{2}+\mfrac{\rho}{2-2\rho}.\]
\end{lemma}

\begin{proof}
Let $g \coloneqq \gcd(W)\, |E(W)|$ if $W$ is a graph and $g \coloneqq 1$ otherwise. It is clear that we must have $\rho < 1$. For a given $n \equiv 0 \mod{4g}$, let $V_1$ and $V_2$ be disjoint vertex sets with $|V_1|=|V_2|=\frac{n}{2}$. Let $h$ be the largest integer such that
$h \equiv 0 \mod{2g}$ and $h < \frac{\rho}{2-2\rho}n$. Let $G_0$ be the vertex disjoint union of an $h$-regular graph with vertex set $V_1$ and an $h$-regular graph with vertex set $V_2$, let $G_1$ be the complete bipartite graph with bipartition $\{V_1,V_2\}$, and let $G$ be the edge disjoint union of $G_0$ and $G_1$. Observe that $G$ is $(\frac{n}{2}+h)$-regular and hence, if $W$ is a graph, then $G$ is $W$-divisible since $n \equiv 0 \mod{4g}$ and $h \equiv 0 \mod{2g}$. For any real $\epsilon >0$, we can take $n$ to be sufficiently large that  $h \geq (\frac{\rho}{2-2\rho}-\epsilon)n$ and hence that the valency of $G$ is at least $(\frac{1}{2}+\frac{\rho}{2-2\rho}-\epsilon)n$. So it suffices to show that $G$ does not have a fractional decomposition into copies of $W$. By our hypotheses, any scaled copy $W'$ of $W$ in $G$ must have weight at least $\rho \Vert W' \Vert$ on edges of $G_0$ and hence weight at most $(1-\rho) \Vert W' \Vert$ on edges of $G_1$. Thus there cannot be a fractional $W$-decomposition of $G$ because
\[\mfrac{|E(G_0)|}{|E(G_1)|} = \mfrac{\frac{1}{2}hn}{\frac{1}{4}n^2} <  \mfrac{\rho}{1-\rho},\] where we used $h < \frac{\rho}{2-2\rho}n$ to obtain the inequality.
\end{proof}

Taking $W$ to be $T_{e_1,e_2,e_3}$ and $\rho=\frac{e_3}{e_1+e_2+e_3}$ in Lemma~\ref{L:lowerBound} shows that the first expression in the maximum in Theorem~\ref{T:weightedTriDecomp} cannot be improved beyond $\frac{e_3}{2e_1+2e_2}$. Also,
taking $W$ to be an $\ell$-cycle and $\rho=\frac{1}{\ell}$ in Lemma~\ref{L:lowerBound}, we have the following well-known lower bound for the decomposition threshold of an odd-length cycle.

\begin{corollary}
For each odd $\ell \geq 3$, we have $\delta^*_{C_{\ell}} \geq \frac{1}{2}+\frac{1}{2\ell-2}$.
\end{corollary}

Our second lower bound is based on a different partition density condition. A proper $k$-colouring of a weighted graph is simply a proper $k$-colouring of its underlying graph.

\begin{lemma}\label{L:lowerBound2}
Let $\rho<1$ be a real number and let $W$ be a nonempty properly $4$-colourable weighted graph such that, for every
proper $4$-colouring of $W$ with colour classes $\{U_1,U_2,U_3,U_4\}$, we have $\Vert W[U_1,U_2] \cup W[U_3,U_4] \Vert \leq \rho\Vert W \Vert$. Then
\begin{equation}\label{E:lowerBound2}
\delta^*_{W} \geq \mfrac{1}{4}\left(3-\sqrt{\mfrac{3\rho-1}{1+\rho}}\,\right) \geq \mfrac{1}{2}+\mfrac{1-\rho}{2+6\rho}.
\end{equation}
\end{lemma}

\begin{proof}
For any proper $4$-colouring of $W$, we may permute the colours in such a way that $\Vert W[U_1,U_2] \cup W[U_3,U_4] \Vert \geq \Vert W[U_1,U_3] \cup W[U_2,U_4] \Vert \geq \Vert W[U_1,U_4] \cup W[U_2,U_3] \Vert$ and hence $\Vert W[U_1,U_2] \cup W[U_3,U_4] \Vert \geq \frac{1}{3}\Vert W \Vert$. Thus it must be that $\rho \geq \frac{1}{3}$.
Observe that the second inequality in \eqref{E:lowerBound2} is equivalent to $(\frac{3\rho-1}{1+\rho})^{1/2} \leq \frac{5\rho-1}{1+3\rho}$ and hence, squaring both sides and simplifying, to the true statement $(1-\rho)^3 \geq 0$. So it remains to prove the first inequality in \eqref{E:lowerBound2}.

Let $g=\gcd(W)\,|E(W)|$ if $W$ is a graph and $g=1$ otherwise. Let $\gamma \coloneqq \frac{1}{4}(1+(\frac{3\rho-1}{1+\rho})^{1/2})$ and note that $\frac{1}{4} \leq \gamma < \frac{1}{2}$ and that the second expression in \eqref{E:lowerBound2} is $1-\gamma$. Let $n \equiv 0 \mod{4g}$ be large and let $h$ be the smallest integer such that $h \equiv 0 \mod{2g}$ and $h > \gamma n$. Take $G$ to be the complete 4-partite graph with partition $\{V_1,V_2,V_3,V_4\}$ where $|V_1|=|V_2|=h$ and $|V_3|=|V_4|=\frac{n}{2}-h$. Observe that each vertex in $G$ has degree in $\{\frac{n}{2}+h,n-h\}$ and hence, if $W$ is a graph then $G$ is $W$-divisible since $n \equiv 0 \mod{4g}$ and $h \equiv 0 \mod{2g}$. For any real $\epsilon >0$, we can take $n$ to be sufficiently large that $h \leq (\gamma+\epsilon)n$ and hence that the minimum degree of $G$ is at least $(1-\gamma-\epsilon)n$. So it suffices to show that $G$ does not have a fractional decomposition into copies of $W$. Let $G_0=G[V_1,V_2] \cup G[V_3,V_4]$ and $G_1=G[V_1 \cup V_2,V_3 \cup V_4]$ so that $G$ is the edge-disjoint union of $G_0$ and $G_1$. By our hypotheses, any scaled copy $W'$ of $W$ in $G$ must have weight at most $\rho \Vert W' \Vert$ on edges of $G_0$ and hence weight at least $(1-\rho) \Vert W' \Vert$ on edges of $G_1$. Thus there cannot be a fractional $W$-decomposition of $G$ because
\[\mfrac{|E(G_0)|}{|E(G_1)|} = \mfrac{h^2+(\frac{n}{2}-h)^2}{2h(n-2h)} = \mfrac{n^2}{8h(n-2h)}-\mfrac{1}{2} > \mfrac{1}{8\gamma(1-2\gamma)}-\mfrac{1}{2} = \mfrac{\rho}{1-\rho},\]
where the inequality is obtained by substituting $h > \gamma n$, noting that $\frac{n^2}{8x(n-2x)}$ is increasing in $x$ on the interval $[\frac{n}{4},\frac{n}{2})$.
\end{proof}

Taking $W$ to be $T_{e_1,e_2,e_3}$ and $\rho=\frac{e_1}{e_1+e_2+e_3}$ in the weaker bound in Lemma~\ref{L:lowerBound2} shows that the second expression in the maximum in Theorem~\ref{T:weightedTriDecomp} cannot be improved beyond $\frac{e_2+e_3}{8e_1+2e_2+2e_3}$.

\section{Fractional decompositions into weighted triangles}\label{S:weightedTriDecomps}

In this section we prove Theorem~\ref{T:weightedTriDecomp}. The following preliminary lemma is a straightforward consequence of two classical results on convexity: Hardy, Littlewood and P\'olya's theorem on majorization \cite{HarLitPol} and Birkhoff's theorem on doubly-stochastic matrices \cite{Bir1946}. Note that a \emph{convex combination} of vectors is a linear combination in which all the coefficients are nonnegative and sum to $1$.

\begin{lemma}[\cite{Bir1946, HarLitPol}; see also \cite{Rus1981}]
\label{L:permuted-convex}
Suppose $a_1 \ge \dots \ge a_n \ge 0$ and $b_1 \ge \cdots \ge b_n \ge 0$ are real numbers satisfying $\sum_{i=1}^k a_i \ge \sum_{i=1}^k b_i$ for each $k=1,\dots,n-1$, and $\sum_{i=1}^n a_i = \sum_{i=1}^n b_i$.  Then $(b_1,\dots,b_n)$ is a convex combination of the vectors in $\{(a_{\pi(1)},\dots,a_{\pi(n)}):\pi \in S_n\}$, where $S_n$ is the set of all permutations of $\{1,\ldots,n\}$.
\end{lemma}

Briefly, if we let $\vec{a}=(a_1,\dots,a_n)$ and $\vec{b}=(b_1,\dots,b_n)$ be vectors as in Lemma~\ref{L:permuted-convex}, then from Hardy-Littlewood-P\'olya there exists a doubly-stochastic matrix $P$ such that $P\vec{a}=\vec{b}$, and from Birkhoff's theorem, $P$ is a convex combination of permutation matrices. In what follows, we make use of the case $n=3$, where the conditions $\sum_{i=1}^k a_i \ge \sum_{i=1}^k b_i$ for $k=1,2$ and $\sum_{i=1}^n a_i = \sum_{i=1}^n b_i$ reduce to $a_1 \ge b_1$, $b_3 \ge a_3$, and $a_1+a_2+a_3=b_1+b_2+b_3$.

Recall that, for any positive real numbers $a$, $b$ and $c$, we let $T_{a,b,c}$ denote the weighted $K_3$ whose edges have weights $a$, $b$ and $c$. We call such graphs \emph{weighted triangles}. Our next result characterises when there exists a fractional decomposition of one weighted triangle into scaled copies of another.

\begin{theorem}\label{T:weightedTriIntoWeightedTri}
Let $w_1$, $w_2$, $w_3$, $e_1$, $e_2$, $e_3$ be positive real numbers such that $w_1 \geq w_2 \geq w_3$ and $e_1 \geq e_2 \geq e_3$. Then $T_{w_1,w_2,w_3}$ has a fractional $(T_{e_1,e_2,e_3})$-decomposition if and only if
\begin{equation}\label{E:goodTest}
\mfrac{w_3}{w_1+w_2+w_3} \ge \mfrac{e_3}{e_1+e_2+e_3}\quad\text{ and }\quad\mfrac{w_1}{w_1+w_2+w_3} \le \mfrac{e_1}{e_1+e_2+e_3}.
\end{equation}
\end{theorem}

\begin{proof}
We first prove the `only if' direction. Considering averages, a fractional $(T_{e_1,e_2,e_3})$-decomposition of $T_{w_1,w_2,w_3}$ must contain weighted graphs $T'$ and $T''$ such that $T'$ has weight at most $\frac{w_3}{w_1+w_2+w_3} \Vert T' \Vert$ on the edge of weight $w_3$ and $T''$ has weight at least $\frac{w_1}{w_1+w_2+w_3} \Vert T'' \Vert$ on the edge of weight $w_1$ in $T_{w_1,w_2,w_3}$. So, since $T'$ and $T''$ are scaled copies of $T_{e_1,e_2,e_3}$, their existence implies, respectively, that the first and second conditions of \eqref{E:goodTest} hold.

Now we prove the `if' direction. Suppose the conditions of \eqref{E:goodTest} hold. Apply Lemma~\ref{L:permuted-convex} in the case $n=3$ with $a_i$ taking the role of $\frac{e_i}{e_1+e_2+e_3}$ and $b_i$ taking the role of $\frac{w_i}{w_1+w_2+w_3}$. Then the conditions of \eqref{E:goodTest} tell us that $b_3 \ge a_3$ and $b_1 \le a_1$.  Thus, by Lemma~\ref{L:permuted-convex}, $(b_1,b_2,b_3)$ is a convex combination of the six permuted copies of $(a_1,a_2,a_3)$.  It follows by scaling that $(w_1,w_2,w_3)$ is a nonnegative linear combination of the six permuted copies of $(e_1,e_2,e_3)$, each of which can be induced by a permutation of the three vertices of a weighted triangle $T_{e_1,e_2,e_3}$.
\end{proof}

\begin{remark}
If desired, explicit coefficients for a fractional $(T_{e_1,e_2,e_3})$-decomposition of $T_{w_1,w_2,w_3}$ can be found by solving a system of linear equations in nonnegative rationals.
\end{remark}

Note that Theorem~\ref{T:weightedTriIntoWeightedTri} implies that, for fixed $e_1 \geq e_2 \geq e_3$, a graph has a fractional $(T_{e_1,e_2,e_3})$-decomposition if and only if it has a fractional decomposition into weighted triangles each of which is, for some $w_1 \geq w_2 \geq w_3$ satisfying \eqref{E:goodTest}, a copy of $T_{w_1,w_2,w_3}$. Let $\ell \geq 3$ be an integer. In the case $(e_1,e_2,e_3)=(\ell-2,1,1)$, the inequalities in \eqref{E:goodTest} reduce to $(\ell-1)w_3 \ge w_1+w_2$ and $2w_1 \le (\ell-2)(w_2+w_3)$, and it is not difficult to see that the first of these implies the second.

We will use Theorem~\ref{T:weightedTriIntoWeightedTri} to prove Theorem~\ref{T:weightedTriDecomp}. For any edge $xy$ of a graph $G$, let $t_{xy}$ denote the number of vertices in $V(G) \setminus \{x,y\}$ that are adjacent to both $x$ and $y$.
We first prove some simple facts about the values of $t_{xy}$ in graphs with specified minimum degree.

\begin{lemma}\label{L:tRatio}
Let $\gamma$ be a positive real number with $\gamma < \frac{1}{2}$ and let $G$ be a graph with $n$ vertices and minimum degree at least $(1-\gamma)n$. For any adjacent edges $e$ and $e'$ in $G$ we have
\[\mfrac{1-2\gamma}{1-\gamma} \leq \mfrac{t_{e}}{t_{e'}} \leq \mfrac{1-\gamma}{1-2\gamma}\]
\end{lemma}

\begin{proof}
Without loss of generality it suffices to suppose that $t_{e'} \leq t_{e}$ and prove that $\frac{t_{e}}{t_{e'}} \leq \frac{1-\gamma}{1-2\gamma}$. Since $t_{e'} \geq (1-2\gamma)n$, this holds if $t_{e}\leq (1-\gamma)n$. Otherwise $t_{e} > (1-\gamma)n$ and $t_{e'} \geq t_e-\gamma n$ since $e$ and $e'$ are adjacent. So we have
\[\mfrac{t_{e}}{t_{e'}} \leq \mfrac{t_{e}}{t_e-\gamma n} < \mfrac{1-\gamma}{1-2\gamma}\]
where the last inequality holds because $t_{e} > (1-\gamma)n$ and $\frac{x}{x-\gamma n}$ is decreasing in $x$.
\end{proof}

\begin{lemma}\label{L:tRatio2}
Let $\gamma$ be a positive real number with $\gamma < \frac{1}{2}$ and let $G$ be a graph with $n$ vertices and minimum degree at least $(1-\gamma)n$. Suppose there is a copy of $K_3$ in $G$ with vertex set $\{x,y,z\}$ such that $t_{xy} \leq t_{xz} \leq t_{yz}$. Then
\begin{itemize}
    \item[\textup{(i)}]
$\frac{1}{t_{yz}} \geq \frac{1-2\gamma}{2-2\gamma}(\frac{1}{t_{xy}}+\frac{1}{t_{xz}})$
    \item[\textup{(ii)}]
$\frac{1}{t_{xy}} \leq \frac{1-\gamma}{2-4\gamma}(\frac{1}{t_{xz}}+\frac{1}{t_{yz}})$
\end{itemize}
\end{lemma}

\begin{proof}
By Lemma~\ref{L:tRatio}, we have
\[\mfrac{t_{yz}}{t_{xy}}+\mfrac{t_{yz}}{t_{xz}} \leq \mfrac{2(1-\gamma)}{1-2\gamma} \quad\text{ and }\quad\mfrac{t_{xy}}{t_{xz}}+\mfrac{t_{xy}}{t_{yz}} \geq \mfrac{2(1-2\gamma)}{1-\gamma}.\]
Rearranging the first of these yields (i) and rearranging the second yields (ii).
\end{proof}

We can now prove Theorem~\ref{T:weightedTriDecomp}.

\begin{proof}[\textup{\textbf{Proof of Theorem~\ref{T:weightedTriDecomp}.}}]
Let $\gamma \coloneqq 1-\delta(e_1,e_2,e_3)$. Suppose $G$ is a graph with $n$ vertices and minimum degree at least $(1-\gamma)n$. Let $\mathcal{K}$ be the set of copies of $K_3$ in $G$. For each $K \in \mathcal{K}$, let $W_K$ be the weighted triangle on vertex set $V(K)$ in which each edge $e \in E(K)$ has weight $\frac{1}{t_{e}}$. Since each edge $e$ is in exactly $t_e$ elements of $\mathcal{K}$, it follows that $\mathcal{G^*}=\{W_K:K \in \mathcal{K}\}$ is a fractional decomposition  of $G$. It now suffices to show that each $W \in \mathcal{G^*}$ has a fractional $(T_{e_1,e_2,e_3})$-decomposition.

Fix $W \in \mathcal{G^*}$ and suppose that $W$ is a copy of $T_{w_1,w_2,w_3}$ with $w_1 \geq w_2 \geq w_3$. By Lemma~\ref{L:tRatio2}, $w_3 \geq \frac{1-2\gamma}{2-2\gamma}(w_1+w_2)$ and $w_1 \leq \frac{1-\gamma}{2-4\gamma}(w_2+w_3)$. Using $\gamma \leq \frac{1}{2}-\frac{e_3}{2e_1+2e_2-2e_3}$, the first of these gives us $\frac{w_3}{w_1+w_2} \geq \frac{e_3}{e_1+e_2}$. Using $\gamma \leq \frac{1}{2}-\frac{e_2+e_3}{8e_1-2e_2-2e_3}$, the second of these gives us $\frac{w_1}{w_2+w_3} \leq \frac{e_1}{e_2+e_3}$. Taking reciprocals, adding one to each side, and taking reciprocals again, we see that these facts are equivalent to the conditions of \eqref{E:goodTest}. So by Theorem~\ref{T:weightedTriIntoWeightedTri}, $W$ has a fractional $(T_{e_1,e_2,e_3})$-decomposition and the proof is complete.
\end{proof}

\section{Approximate decompositions from condensations}\label{S:condensations}

Our goal in this section is to prove Theorem~\ref{T:approximateCycleDecomp}. Our arguments follow those used in \cite{HaxRod} and \cite{Yus2005}.  Let $G$ be a graph and let $A$ and $B$ be nonempty disjoint subsets of $V(G)$. The \emph{density} $d(G[A,B])$ of $G[A,B]$ is defined to be $\frac{1}{|A||B|}|E(G[A,B])|$. For a positive real $\epsilon$, we say that $G[A,B]$ is \emph{$\epsilon$-regular} if, for every $X \subseteq A$ and $Y \subseteq B$ with $|X| \geq \epsilon|A|$ and $|Y| \geq \epsilon|B|$, we have $|d(G[X,Y])-d(G[A,B])| \leq \epsilon$. If we also wish to specify the density of $G[A,B]$, we will say it is \emph{$(\epsilon,d(G[A,B]))$-regular}. Further, we say it is \emph{$(\epsilon,\delta \pm \alpha)$-regular} if it is \emph{$\epsilon$-regular} and $\delta-\alpha \leq  d(G[A,B])) \leq \delta+\alpha$.

Let $F$ be a graph on vertex set $\{1,\ldots,k\}$. For an integer $b$ and real numbers $\delta$ and $\epsilon$, we say a graph $H$ is an \emph{$(F,b,\delta,\epsilon)$-graph} if there is a partition $\{V_1,\ldots,V_k\}$ of $V(H)$ such that
\begin{itemize}[nosep]
    \item
$|V_1|=\cdots=|V_k|=b$;
    \item
for all distinct $i,j \in \{1,\ldots,k\}$, $H[V_i \cup V_j]$ is empty if $ij \notin E(F)$; and
    \item
for all distinct $i,j \in \{1,\ldots,k\}$, $H[V_i,V_j]$ is $(\epsilon,\delta\pm\epsilon)$-regular if $ij \in E(F)$.
\end{itemize}
So, roughly speaking, an $(F,b,\delta,\epsilon)$-graph is a graph obtained from $F$ by blowing up each vertex into an independent set of size $b$ and each edge into an $\epsilon$-regular bipartite graph of density approximately $\delta$. We say a graph is an \emph{$(F,b,\geq \delta_0,\epsilon)$-graph} if it is an $(F,b, \delta,\epsilon)$-graph for some $\delta \geq \delta_0$. We will make use of the following result from \cite{HaxRod}.

\begin{lemma}[{\cite[Lemma 5]{HaxRod}}]\label{L:HaxRodFinalDecomp}
Let $F$ be a graph and let $\eta$ and $\delta_0$ be given positive reals. There exist $\epsilon \coloneqq \epsilon(F,\eta,\delta_0)$ and $b_0 \coloneqq b_0(F,\eta,\delta_0)$ such that, if $H$ is an $(F,b,\geq \delta_0,\epsilon)$-graph with $b \geq b_0$, then there is an  $F$-packing of $H$ whose leftover has size at most $\eta|E(H)|$.
\end{lemma}

By Lemma~\ref{L:HaxRodFinalDecomp}, it will suffice to find approximate decompositions into $(F,b,\geq\delta_0,\epsilon)$-graphs. Thus our main goal in this section will be to prove the following lemma.

\begin{lemma}\label{L:initialDecomp}
Let $F$ be a graph, let $\mathcal{W}$ be a set of condensations of $F$, and let $\eta$ be a positive real number. There are positive reals $\delta_0 \coloneqq\delta_0(F,\eta)$ and $\epsilon_0 \coloneqq \epsilon_0(F,\eta)$ such that the following holds. For each positive real $\epsilon<\epsilon_0$, there are integers $n_0 \coloneqq n_0(F,\eta,\epsilon)$ and $t_0 \coloneqq t_0(F,\eta,\epsilon)$ such that each graph $G$ with $n>n_0$ vertices that has a fractional $\mathcal{W}$-decomposition also has, for some $b \geq \frac{n}{t_0}$, an $\eta$-approximate decomposition with $(F,b,\geq\delta_0,\epsilon)$-graphs.
\end{lemma}

A result along the lines of Lemma~\ref{L:initialDecomp} is implicitly proved (but not directly stated) in \cite{HaxRod} and \cite{Yus2005}. In those results, however, the hypothesis that $G$ has a fractional $W$-decomposition is replaced by the hypothesis that $G$ has a fractional $F$-decomposition. Our proof of Lemma~\ref{L:initialDecomp} proceeds along similar lines to the proofs in \cite{HaxRod} and \cite{Yus2005}. The only substantially new idea we require is encapsulated in Lemma~\ref{L:blowUpDecomp} below.

For a weighted graph $W$ and a positive integer $q$, the \emph{$q$-blow up} of $W$ is defined to be the weighted graph $Q$ with $V(Q)=V(W) \times \{1,\ldots,q\}$ in which, for all $i,j \in \{1,\ldots,q\}$, the edge $\{(x,i),(y,j)\}$ is present and has weight $w_W(xy)$ if $xy \in E(W)$ and is absent otherwise. So, informally, $Q$ is obtained from $W$ by blowing up each vertex into $q$ vertices and each edge into a scaled copy of $K_{q,q}$ each of whose edges has the weight of the original edge.

\begin{lemma}\label{L:blowUpDecomp}
Let $F$ be a graph, let $W$ be the condensation of $F$ with respect to a proper $k$-colouring with colour classes $\{U_1,\ldots,U_k\}$, and let $q \geq \max\{|U_1|,\ldots,|U_k|\}$. The $q$-blow up of $W$ has a fractional $F$-decomposition.
\end{lemma}

\begin{proof}
Let $Q$ be the $q$-blow up of $W$ and note that $V(Q)=\{1,\ldots,k\} \times \{1,\ldots,q\}$. Say an injection $a$ from $V(F)$ to $V(Q)$ is \emph{partition respecting} if, for all $i \in \{1,\ldots,k\}$ and $x \in U_i$, we have $a(x) = (i,j)$ for some $j \in \{1,\ldots,q\}$. Let $A$ be the set of all partition respecting injections and note that $A$ is nonempty since $q \geq \max\{|U_1|,\ldots,|U_k|\}$.

For each $a \in A$, let $Q_a$ be the image of $F$ under $a$ with scale factor $\frac{1}{|A|}q^2$ and let $\mathcal{Q}^*=\{Q_a:a\in A\}$. Obviously each graph in $\mathcal{Q}^*$ is a weighted subgraph of $Q$. Let $ii'$ be an arbitrary edge of $W$ and let $E_{ii'}$ be the subset $\{\{(i,j),(i',j')\}:j,j'\in\{1,\ldots,q\}\}$ of $E(Q)$. There are $w_W(ii')$ edges of $F$ between $U_i$ and $U_{i'}$ and thus
\[\sum_{e \in E_{ii'}}\sum_{Q_a \in \mathcal{Q}^*}w_{Q_a}(e)=\sum_{a \in A}\mfrac{q^2}{|A|}w_W(ii')=q^2w_W(ii').\]
Since $|E_{ii'}|=q^2$, we have by symmetry that $\sum_{Q_a \in \mathcal{Q^*}}w_{Q_a}(e)=w_W(ii')=w_Q(e)$ for each $e \in E_{ii'}$. So $\mathcal{Q}^*$ is a fractional $F$-decomposition of $Q$.
\end{proof}

In order to prove Lemma~\ref{L:initialDecomp} we introduce some tools from elsewhere. A partition is \emph{equitable} if any two of its parts differ in size by at most 1. We use a version of the regularity lemma \cite{Sze}. The form below follows easily from \cite[Chapter IV, Theorem $29'$]{Bol}, for example.

\begin{lemma}\label{L:regularity}
Let $\epsilon>0$ be a real number and $m_0$ be an integer. There is an integer $m_1 \coloneqq m_1(\epsilon,m_0)$ such that, for every graph $G$ of order $n>m_1$, there is an equitable partition $\{V_1,\ldots,V_m\}$ of $V(G)$ such that $m_0 \leq m \leq m_1$ and $\sum_{\{i,j\}\in I} |V_i||V_j| \leq \epsilon n^2$ where $I$ is the set of all $2$-subsets $\{i,j\}$ of $\{1,\ldots,m\}$ for which $G[V_i,V_j]$ is not $\epsilon$-regular.
\end{lemma}

An $\epsilon$-regular bipartite graph can be decomposed into subgraphs that retain the regularity property. This can be established using a random partition. The following formulation is from \cite{GirGraKuhOst}. When we write $0 < a \ll b \ll c \leq 1$ in a statement or argument, we mean that there exist non-decreasing functions $f : (0, 1] \rightarrow (0, 1]$ and $g: (0, 1] \rightarrow (0, 1]$ such that the statement or argument holds for all $0 < a, b, c \leq 1$ satisfying $b \leq f(c)$ and $a \leq g(b)$. Hierarchies with more constants are defined similarly.

\begin{lemma}[{\cite[Lemma 4.8]{GirGraKuhOst}}]\label{L:colourRegularPair}
Assume $0 < \frac{1}{b} \ll \epsilon \ll \delta_1,\ldots,\delta_\ell \leq \delta \leq 1$ and suppose that $\sum_{i=1}^\ell \delta_i \leq \delta$. Let $G$ be an $(\epsilon,\delta)$-regular bipartite graph with partition $\{U,V\}$ where $|U|=|V|=b$. Then $G$ can be decomposed into spanning subgraphs $G_0,\ldots,G_\ell$ such that $G_0$ is empty if $\sum_{i=1}^\ell \delta_i = \delta$ and $G_i$ is $(\epsilon^{1/12},\delta_i \pm \epsilon^{1/12})$-regular for each $i \in \{1,\ldots,\ell\}$.
\end{lemma}

We also use the following result from \cite{HaxRod}.

\begin{lemma}[{\cite[Lemma 7]{HaxRod}}]\label{L:HaxRodTauBounded}
Let $F$ be a graph and $\eta$ be a positive real number. There exists a positive real $\delta_0 \coloneqq \delta_0(F,\eta)$ such that, for any weighted graph $Q$ in which each edge has weight at most $1$,
if $Q$ has a fractional $F$-decomposition, then it also has an $\eta$-approximate fractional $F$-decomposition in which the weight of each edge in each copy of $F$ is at least $\delta_0$.
\end{lemma}

We are now in a position to prove Lemma~\ref{L:initialDecomp}.

\begin{proof}[\textup{\textbf{Proof of Lemma~\ref{L:initialDecomp}}}]
By Lemma~\ref{L:HaxRodTauBounded}, take $\delta_0 \coloneqq \delta_0(F,\eta)$ such that every weighted graph that has a fractional $F$-decomposition also has an $\frac{\eta}{4}$-approximate fractional $F$-decomposition in which the weight of each edge of each copy of $F$ is at least $2\delta_0$. Let $q=|V(F)|$. Take $\epsilon_0 \coloneqq \frac{1}{8}\eta\delta_0$ and let $\epsilon< \epsilon_0$ be a positive real. Take $\epsilon_1$ and $\epsilon_2$ to be positive reals such that $0 < \epsilon_1 \ll \epsilon_2 \ll \epsilon <1$. Let $m_1 \coloneqq m_1(\epsilon_1,\lceil\frac{8}{\eta}q^2\rceil)$ be as given by Lemma~\ref{L:regularity} and let $t_0 \coloneq \lceil \frac{8}{\eta}q^3m_1 \rceil$. Finally take $G$ to be a graph of order $n$ where $n$ is large relative to $t_0$ and $\frac{1}{\epsilon_2}$, and assume there is a fractional $\mathcal{W}$-decomposition $\mathcal{G}^*$ of $G$.

By Lemma~\ref{L:regularity} there is an equitable partition $\{V'_1,\ldots,V'_{m'}\}$ of $V(G)$ such that $\frac{1}{\eta}8q^2 \leq m' \leq m_1$ and $\sum_{\{i,j\}\in I} |V'_i||V'_j| \leq \epsilon_1 n^2$ where $I$ is the set of all 2-subsets $\{i,j\}$ of $\{1,\ldots,m'\}$ for which $G[V'_i,V'_j]$ is not $\epsilon_1$-regular. Let $b=\lfloor \frac{\eta n}{4q^3m'} \rfloor$ and note that $b > n/t_0$. Our first step is to refine the partition $\{V'_1,\ldots,V'_{m'}\}$ into a new one with more desirable properties. We claim we can find a partition $\{V_0,\ldots,V_{m}\}$ of $V(G)$ such that
\begin{itemize}[nosep]
    \item
$|V_1|=\cdots=|V_{m}|=bq$ and $|V_0|<bqm'$;
    \item
for each $j \in \{1,\ldots,m'\}$, $|V'_j \cap V_0| <bq$ and $V'_j \setminus V_0$ is a union of classes in $\{V_1,\ldots,V_{m}\}$;
    \item
the sum of the weights of the weighted graphs in $\mathcal{G}^*$ that have two or more vertices in any part in $\{V_1,\ldots,V_{m}\}$ is at most $\frac{\eta}{8}n^2$.
\end{itemize}
Consider randomly partitioning each part in $\{V'_1,\ldots,V'_{m'}\}$ into some number of parts of size $bq$ and one (possibly trivial) part of size less than $bq$. Let $V_1,\ldots,V_{m}$ be the resultant parts of size $bq$ and $V_0$ be the union of all the parts of size less than $bq$. Clearly $|V_0|<bqm'$. Since $|V'_i| \geq \lfloor\frac{n}{m'}\rfloor$ for each $i \in \{1,\ldots,m'\}$ the probability that a given pair of vertices lie together in a part in $\{V_1,\ldots,V_{m}\}$ is certainly less than $bq/\lfloor\frac{n}{m'}\rfloor \leq \frac{2}{n}bqm'$ (using $n \gg m'$). Thus, using the union bound, the probability that a given weighted graph in $\mathcal{G}^*$ has two or more vertices in any part in $\{V_1,\ldots,V_{m}\}$ is less than
\[\mfrac{2bqm'}{n}\mbinom{|V(W)|}{2} < \mfrac{bq^3m'}{n} \leq \mfrac{\eta}{4}\]
using the definition of $b$. Thus the expected sum of the weights of the weighted graphs in $\mathcal{G}^*$ that have two or more vertices in any part in $\{V_1,\ldots,V_{m}\}$ is less than $\frac{\eta}{4}|E(G)| < \frac{\eta}{8}n^2$.
So our claim follows.

Our next step is to restrict our attention to only those weighted graphs in $\mathcal{G}^{*}$ for which every edge is in an $\epsilon_2$-regular bipartite graph and show that, in doing so, we do not lose too much weight. Let $\mathcal{G}^{**}$ be the fractional packing of $G$ that contains every weighted graph in $\mathcal{G}^*$ except those which:
\begin{itemize}[nosep]
    \item
have two or more vertices in some part in $\{V_1,\ldots,V_{m}\}$, or
    \item
have an edge in a bipartite graph $G[V_i,V_j]$ that is not $\epsilon_2$-regular.
\end{itemize}
We may assume $\epsilon_2 > \frac{8}{\eta}q^2\epsilon_1$. For all distinct $i,j \in \{0,\ldots,m\}$, if $V_i$ and $V_j$ are subsets of distinct parts of an $\epsilon_1$-regular bipartite graph then, by the definition of $\epsilon_1$-regular, $G[V_i,V_j]$ is $(\frac{1}{bq}\lceil\frac{n}{m'}\rceil\epsilon_1)$-regular. Using the definition of $b$, we have $\frac{1}{bq}\lceil\frac{n}{m'}\rceil\epsilon_1 <  \frac{8}{\eta}q^2\epsilon_1 < \epsilon _2$. Thus $G[V_i,V_j]$ is $\epsilon_2$-regular unless one of the following holds.
\begin{itemize}[nosep]
    \item
$0 \in \{i,j\}$. At most $bqm'n$ edges can be in such graphs.
    \item
$V_i \cup V_j \subseteq V'_h$ for some $h \in \{1,\ldots,m'\}$. At most $m'\binom{\lceil n/m' \rceil}{2}$ edges can be in such graphs.
    \item
$V_i \subseteq V'_g$ and $V_j \subseteq V'_h$ for some distinct $g,h \in \{1,\ldots,m'\}$ such that $G[V'_g,V'_h]$ is not $\epsilon_1$-regular. At most $\epsilon_1 n^2$ edges can be in such graphs.
\end{itemize}
So the number of edges in graphs $G[V_i,V_j]$ that are not $\epsilon_2$-regular is at most
\[\epsilon_1 n^2+m'\mbinom{\lceil n/m' \rceil}{2}+bqm'n \leq \left(\epsilon_1+\mfrac{1}{m'}+\mfrac{\eta}{4q^2}\right)n^2 \leq \mfrac{\eta}{2q^2}n^2,\]
using $\epsilon_1,\frac{1}{m'} < \frac{\eta}{8q^2}$. The sum of the weights of graphs in $\mathcal{G}^{*}$ that use these edges can be at most $|E(F)|$ times this quantity, and $|E(F)| \leq \binom{q}{2}$. Hence,
\begin{equation}\label{E:GDoubleStar}
\Vert\mathcal{G}^{**}\Vert > |E(G)|-\mfrac{\eta}{8}n^2-\mbinom{q}{2}\mfrac{\eta}{2q^2}n^2 > |E(G)|-\mfrac{\eta}{2}n^2.
\end{equation}

We will now define a weighted graph that records the `density' of weight used by weighted graphs in $\mathcal{G}^{**}$ between the partition classes in $\{V_1,\ldots,V_{m}\}$ and thus find a fractional $\mathcal{W}$-decomposition of this graph by amalgamating $\mathcal{G}^{**}$ in the natural way. We can then use Lemma~\ref{L:blowUpDecomp} to find a fractional $F$-decomposition of a blow up of this graph. Let $R$ be the weighted graph with vertex set $\{1,\ldots,m\}$ in which edge $ij$ is absent if no graph in $\mathcal{G}^{**}$ contains an edge of $G[V_i,V_j]$ and otherwise $w_R(ij)$ is $\frac{1}{b^2q^2}$ times the sum of the total weight assigned to the edges in $E(G[V_i,V_j])$ by the weighted graphs in $\mathcal{G}^{**}$. Note that $\Vert R\Vert=\frac{1}{b^2q^2}\Vert\mathcal{G}^{**}\Vert$ and that, if $ij \in E(R)$, then $G[V_i,V_j]$ is $\epsilon_2$-regular and $ij$ has weight at most $d(G[V_i,V_j]) \leq  1$.  We define a fractional $\mathcal{W}$-decomposition $\mathcal{R}^*$ of $R$ as follows. For each $W' \in \mathcal{G}^{**}$ we add to $\mathcal{R}^*$ a weighted graph that is similar to $W'$ with scale factor $\frac{1}{b^2q^2}$ under an isomorphism that maps each $x \in V(W')$ to the $i \in \{1,\ldots,m\}$ such that $x \in V_i$.

By applying Lemma~\ref{L:blowUpDecomp} to each weighted graph in $\mathcal{R}^*$, we can obtain a fractional $F$-decomposition of the $q$-blow up $Q$ of $R$. Hence, by our definition of $\delta_0$, there is a fractional $F$-packing $\mathcal{Q}^*=\{Q_1,\ldots,Q_z\}$ of $Q$ such that each edge of each scaled copy of $F$ in $\mathcal{Q}^*$ has weight at least $2\delta_0$ and
\begin{equation}\label{E:QStar}
    \Vert\mathcal{Q}^*\Vert \geq \Vert Q \Vert -\mfrac{\eta}{4}(mq)^2 = q^2\Vert R \Vert-\mfrac{\eta}{4}(mq)^2=\mfrac{1}{b^2}\Vert \mathcal{G}^{**} \Vert-\mfrac{\eta}{4}(mq)^2.
\end{equation}

We complete the proof by refining the partition $\{V_1,\ldots,V_{m}\}$ in such a way that $Q$ approximately (but not exactly) represents the number of edges in $(q\epsilon_2)$-regular bipartite graphs between the refined classes. We will then be able to use Lemma~\ref{L:colourRegularPair} to convert $\mathcal{Q}^*$ into the $\eta$-approximate decomposition of $G$ with $(F,b,\geq\delta_0,\epsilon)$-graphs that we require. Subdivide each partition class $V_i \in \{V_1,\ldots,V_{m}\}$ into $q$ parts $V_{i,1},\ldots,V_{i,q}$, each of size $b$. Notice that if $G[V_i,V_{i'}]$ is $(\epsilon_2,\delta)$-regular, then $G[V_{i,j},V_{i',j'}]$ is $(q\epsilon_2,\delta\pm\epsilon_2)$-regular for all $j$ and $j'$. We now define a packing $\mathcal{G}=\{H_1,\ldots,H_z\}$ of $G$. For each edge $e=\{(i,j),(i',j')\}$ of $Q$ we do as follows. Let $C_{e}$ be the set of integers $h$ in $\{1,\ldots,z\}$ for which the weighted graph $Q_h$ contains edge $e$. Now use Lemma~\ref{L:colourRegularPair} to colour (some of) the edges of $G[V_{i,j},V_{i',j'}]$ with colours in $C_{e}$ in such a way that colour class $h$ induces an $(\epsilon,\delta_h\pm \epsilon)$-regular bipartite graph where $\delta_h=(1-\epsilon_2)w_h$ and $w_h$ is the weight of each edge of $Q_h$.
This is a valid application of Lemma~\ref{L:colourRegularPair} because
\[\sum_{h \in C_{e}}\delta_h \leq (1-\epsilon_2)w_{Q}(e) = (1-\epsilon_2)w_R(ii') \leq (1-\epsilon_2)d(G[V_{i},V_{i'}]) \leq d(G[V_{i,j},V_{i',j'}]).\]
Here we used, in order, the fact that $\mathcal{Q}^*$ is a packing of $Q$, the definition of $Q$, and the definition of $R$. For each $h \in \{1,\ldots, z\}$, let $H_h$ be the graph induced by the edges of colour $h$ and note that $H_h$ is a $(F,b,\delta_h,\epsilon)$-graph because $Q_h$ is a scaled copy of $F$ and, for all $\{(i,j),(i',j')\} \in E(Q_h)$, $H_h[V_{i,j},V_{i',j'}]$ is $(\epsilon,\delta_h\pm \epsilon)$-regular. In particular, since $w_h \geq 2\delta_0$, $H_h$ is a $(F,b,\geq\delta_0,\epsilon)$-graph.

It only remains to show that the graphs in $\mathcal{G}$ use sufficiently many edges. We have
\begin{multline}\label{E:G}
\Vert\mathcal{G}\Vert = \sum_{h=1}^z|E(H_h)| \geq
\sum_{h=1}^z b^2(\delta_h-\epsilon)|E(F)| \geq \left(1-\mfrac{\epsilon}{\delta_0}\right)b^2\sum_{h=1}^z\delta_h|E(F)|= \\\left(1-\mfrac{\epsilon}{\delta_0}\right)b^2\sum_{h=1}^z(1-\epsilon_2)\Vert Q_h \Vert = \left(1-\mfrac{\epsilon}{\delta_0}\right)(1-\epsilon_2) b^2\Vert \mathcal{Q}^* \Vert,
\end{multline}
where in the second inequality we used the fact that $\delta_h \geq \delta_0$ for each $h \in \{1,\ldots,z\}$.
Finally, we observe that, using \eqref{E:QStar} and \eqref{E:GDoubleStar},
\begin{equation}\label{E:bSquaredQStar}
b^2\Vert \mathcal{Q}^* \Vert \geq \Vert \mathcal{G}^{**} \Vert-\mfrac{\eta}{4}(bmq)^2 \geq \Vert \mathcal{G}^{**} \Vert-\mfrac{\eta}{4}n^2 > |E(G)|-\mfrac{3\eta}{4}n^2.
\end{equation}
Thus, from \eqref{E:G} and \eqref{E:bSquaredQStar}
\[\Vert\mathcal{G}\Vert \geq (1-\epsilon_2)\left(1-\mfrac{\epsilon}{\delta_0}\right)\left(|E(G)|-\mfrac{3\eta}{4}n^2\right) > \left(1-\mfrac{\eta}{4}\right)\left(|E(G)|-\mfrac{3\eta}{4}n^2\right) > |E(G)|-\eta n^2\]
where we used $\epsilon_2 < \epsilon < \frac{1}{8}{\eta\delta_0}$. This completes the proof.
\end{proof}

Theorem~\ref{T:approximateCycleDecomp} can be proved by combining Lemmas~\ref{L:HaxRodFinalDecomp} and \ref{L:initialDecomp}.

\begin{proof}[\textup{\textbf{Proof of Theorem~\ref{T:approximateCycleDecomp}}}]
Fix $F$ and $\eta$, and let $\mathcal{W}$ be a set of condensations of $F$. Let $\delta_0 \coloneqq\delta_0(F,\eta/2)$ and $\epsilon_0 \coloneqq \epsilon_0(F,\eta/2)$ be the constants given by  Lemma~\ref{L:initialDecomp}. Now let $\epsilon \coloneqq \epsilon(F,\eta/2,\delta_0)$ and $b_0 \coloneqq b_0(F,\eta/2,\delta_0)$ be the constants given by Lemma~\ref{L:HaxRodFinalDecomp} and let $\epsilon^*=\min\{\epsilon_0,\epsilon\}$. By Lemma~\ref{L:initialDecomp} there are integers $n_0 \coloneqq n_0(F,\eta/2,\epsilon^*)$ and $t_0 \coloneqq t_0(F,\eta/2,\epsilon^*)$ such that each graph $G$ with $n>n_0$ vertices that has a fractional $\mathcal{W}$-decomposition also has, for some $b \geq \frac{n}{t_0}$, a packing $\mathcal{G}$ with $(F,b,\geq\delta_0,\epsilon)$-graphs whose leftover has size at most $\frac{\eta}{2} n^2$. Assuming $n > \max\{n_0,b_0t_0\}$, then $b \geq b_0$. By Lemma~\ref{L:HaxRodFinalDecomp}, each graph $H \in \mathcal{G}$ has an $F$-packing whose leftover has size at most $\frac{\eta}{2}|E(H)|$. Taking the union of these packings produces an $F$-packing of $G$ whose leftover has size at most
\[\mfrac{\eta n^2}{2}+\sum_{H \in \mathcal{G}}\mfrac{\eta|E(H)|}{2} < \mfrac{\eta n^2}{2}+\mfrac{\eta n^2}{4} < \eta n^2. \qedhere\]
\end{proof}

\section{Decompositions into tripartite graphs}\label{S:tripartite}

In this section, we consider the more general case of $F$-decompositions for tripartite graphs $F$. Note that the decomposition thresholds of bipartite graphs are completely classified in \cite{GloKuhLoMonOst}: all are in $\{0,\frac{1}{2},\frac{2}{3}\}$. It is also shown in \cite{GloKuhLoMonOst} that $\delta_F \leq \delta_{K_{\chi(F)}}$ for any graph $F$ with chromatic number $\chi(F)$. So in particular $\delta_F \leq \delta_{K_{3}} \leq (7+\sqrt{21})/14 \lessapprox 0.827$ for each tripartite graph $F$ using the result of \cite{DelPos}.  We make use of the following result of \cite{GloKuhLoMonOst} which, like Theorem~\ref{T:approxToIntegral}, was proved using iterative absorption methods.

\begin{theorem}[\cite{GloKuhLoMonOst}]
\label{T:approxToIntegralGen}
Let $F$ be a graph with chromatic number $\chi$. Then $\delta_F \leq \max\{\delta_F^{0+},1-\frac{1}{\chi+1}\}$.
\end{theorem}

In conjunction with our results, this means we can say the following about the decomposition thresholds of tripartite graphs.

\begin{theorem}\label{T:genTripartite}
Suppose $F$ is a tripartite graph with tripartition $\{V_1,V_2,V_3\}$ and $F[V_1,V_2]$, $F[V_1,V_3]$ and $F[V_2,V_3]$ have $e_1$, $e_2$ and $e_3$ edges, respectively, where $e_1 \geq e_2 \geq e_3$. Then $\delta_F \leq \max\{\delta(e_1,e_2,e_3),\frac{3}{4}\}$ where $\delta$ is defined as in Theorem~$\ref{T:weightedTriDecomp}$.
\end{theorem}

\begin{proof}
Note that $T_{e_1,e_2,e_3}$ is a condensation of $F$. So, between them, Theorems~\ref{T:weightedTriDecomp} and \ref{T:approximateCycleDecomp} imply that $\delta^{0+}_F \leq \delta(e_1,e_2,e_3)$. So the result follows by applying Theorem~\ref{T:approxToIntegralGen}.
\end{proof}

Note that the bound of Theorem~\ref{T:genTripartite} will be $\frac{3}{4}$ when both $e_3 \leq \frac{1}{3}(e_1+e_2)$ and $e_1 \geq \frac{3}{4}(e_2+e_3)$ hold, and will be $\delta(e_1,e_2,e_3)$ otherwise. We now consider the family of complete tripartite graphs $K_{a,1,1}$. Note that $K_{2,1,1}$ is $K_4^-$, the graph obtained from $K_4$ by removing an edge. We have the following.

\begin{corollary}\phantom{a}
\begin{itemize}
    \item
$0.655 \lessapprox \frac{1}{28}(21-\sqrt{7}) \leq \delta_{K_4^-} \leq \frac{4}{5}$
    \item
$0.605 \lessapprox \frac{1}{12}(9-\sqrt{3}) <\frac{1}{4}(3-(\frac{a-1}{3a+1})^{1/2}) \leq \delta_{K_{a,1,1}} \leq \frac{3}{4}$ for each integer $a \geq 3$.
\end{itemize}
\end{corollary}

\begin{proof}
In both cases the upper bound is obtained by applying Theorem~\ref{T:genTripartite} with $e_1=e_2=a$ and $e_3=1$. The immediate lower bounds are obtained from applying Lemma~\ref{L:lowerBound2} with $\rho=\frac{a}{2a+1}$ (this lower bound is superior to the bound of Lemma~\ref{L:lowerBound} in all these cases).
\end{proof}

Unlike the family of odd-length cycles, we see that the fractional decomposition thresholds of these graphs are bounded away from $\frac{1}{2}$. To see a different behaviour, we consider the family of complete tripartite graphs $K_{a,a,1}$. For these graphs, our results do not rule out the possibility that the decomposition thresholds approach $\frac{1}{2}$ as $a$ becomes large.

\begin{corollary}
$\frac{1}{2}+\frac{1}{2a+2} \leq \delta_{K_{a,a,1}} \leq \frac{3}{4}$ for each integer $a \geq 2$.
\end{corollary}

\begin{proof}
The upper bound is obtained by applying Theorem~\ref{T:genTripartite} with $e_1=a^2$ and $e_2=e_3=a$. The immediate lower bounds are obtained from applying Lemma~\ref{L:lowerBound} with $\rho=\frac{a}{a^2+2a}=\frac{1}{a+2}$ (this lower bound is superior to the bound of Lemma~\ref{L:lowerBound2} in all these cases).
\end{proof}

\section{Conclusion}\label{S:conclusion}

We have explored a method based on condensations, coupled with graph regularity and fractional decompositions into weighted graphs. With this, we have lowered the best known minimum degree threshold sufficient for $F$-decompositions for various tripartite graphs $F$, including $K_4^-$ and odd cycles $C_\ell$.  In the latter case, as $\ell$ grows, our threshold approaches $\frac{1}{2}$ fairly rapidly compared with previous work.

Our results in Section~\ref{S:tripartite} concerning general tripartite graphs were limited because Theorem~\ref{T:approxToIntegralGen} cannot provide bounds better than $\frac{3}{4}$ in the case of tripartite $F$. It would be interesting to investigate whether this limitation can be overcome for at least some tripartite graphs. The method of condensation can generalise to handle graphs $F$ with $\chi(F)=k>3$.  This motivates both the study of fractional decomposition thresholds for weighted copies of $K_k$, as well as the possibility of decreasing the $1-\frac{1}{\chi+1}$ term in Theorem~\ref{T:approxToIntegralGen}.

Another interesting question, even for tripartite $F$, would be whether the absorber methods of \cite{BarKuhLoOstTay} can be adapted to weighted cliques.  If so, this could lead via condensation to degree thresholds for multipartite host graphs $G$.  The fractional side of this problem can be attacked with existing methods, see for instance \cite{Mon}.

\section{Acknowledgements}
Our thanks to Bertille Granet for valuable discussions, and to the anonymous referee who suggested a neater proof of Lemma~\ref{L:blowUpDecomp}. This research was supported by Australian Research Council grant DP240101048. The second author was supported by the Natural Sciences and Engineering Research Council grant RGPIN-2024-03966. The last author was supported by the European Research Council (ERC) under the European Union Horizon 2020 research and innovation programme (grant agreement No.\ 947978). The authors thank the mathematical research institute MATRIX in Australia where the discussions that led to this paper began.


\begin{thebibliography}{99}

    \bibitem{BarKuhLoOst}
B. Barber, D. K\"{u}hn, A. Lo and D. Osthus, Edge decompositions of graphs with high minimum degree, \textit{Adv. Math.} \textbf{288} (2016), 337--385.

    \bibitem{BarKuhLoOstTay}
B. Barber, D. K\"{u}hn, A. Lo, D. Osthus and A. Taylor,
Clique decompositions of multipartite graphs and completion of latin squares,
\textit{J. Combin, Theory Ser. A} \textbf{151} (2017), 146--201.

    \bibitem{Bir1946}
G. Birkhoff,  Tres observaciones sobre el algebra lineal, \textit{Univ. Nac. Tucum\'an Rev. Ser. A} \textbf{5} (1946) 147--151.

    \bibitem{Bol}
B. Bollob\'as, Modern graph theory, \textit{Graduate Texts in Mathematics} \textbf{184}, Springer, New York, 1998.

    \bibitem{DelPos}
M.~Delcourt and L.~Postle, Progress towards Nash-Williams' conjecture on triangle decompositions, \textit{J. Combin. Theory Ser. B} \textbf{146} (2021), 382--416.

    \bibitem{GirGraKuhOst}
A. Gir\~ao, B. Granet, D. K\"uhn and D. Osthus, Path and cycle decompositions of dense graphs,
\textit{J. Lond. Math. Soc.} \textbf{104} (2021), 1085--1134.

    \bibitem{GloKuhLoMonOst}
S. Glock, D. K\"{u}hn, A. Lo, R. Montgomery and D. Osthus, On the decomposition threshold
of a given graph, \textit{J. Combin. Theory Ser. B} \textbf{139} (2019), 47--127.

    \bibitem{HarLitPol}
G.H. Hardy, J.E. Littlewood and G. P\'olya,
Some simple inequalities satisfied by convex functions,
\textit{Messenger Math.} \textbf{58} (1929), 145--152.

    \bibitem{HaxRod}
P. Haxell and V. R\"odl, Integer and fractional packings in dense graphs,
\textit{Combinatorica} \textbf{21} (2001), 13--38.

    \bibitem{JooKuh}
F. Joos and M. K\"uhn, Fractional cycle decompositions in hypergraphs,
\textit{Random Structures Algorithms} \textbf{61} (2022), 425--443.

%     \bibitem{KomSim}
% J. Koml\'os and M. Simonovits,
% Szemerédi’s regularity lemma and its applications in graph theory,
% Bolyai Soc. Math. Stud. 2 (1996), 295--352.

    \bibitem{Mon}
R. Montgomery, Fractional clique decompositions of dense partite graphs, \textit{Combin. Probab. Comput.} \textbf{26} (2017), 911--943.

    \bibitem{Rus1981}
L. R\"uschendorf, Ordering of distributions and rearrangement of functions, \textit{Ann. Probab.} \textbf{9} (1981), 276--283.

    \bibitem{Sze}
E. Szemer\'edi, Regular partitions of graphs, \textit{Colloq. Internat. CNRS} \textbf{260} (1978), 399--401.

    \bibitem{Tay}
A. Taylor, On the exact decomposition threshold for even cycles,
\textit{J. Graph Theory} \textbf{90} (2019), 231--266.

    \bibitem{Yus2005}
R.~Yuster, Integer and fractional packing of families of graphs, \textit{Random Structures Algorithms}
\textbf{26} (2005), 110--118.
\end{thebibliography}
\end{document}